\documentclass[11pt]{amsart}
\usepackage{geometry}                
\usepackage{graphicx}
\usepackage{amssymb}
\usepackage{epstopdf}
\usepackage[all]{xypic}
\usepackage{palatino}

\usepackage{hyperref}
\usepackage[arrow,curve,matrix]{xy}

\renewcommand{\dim}{{\rm dim}}

\newcommand{\rk}{{\rm rk}}

\newcommand{\Ext}{{\rm Ext}}

\theoremstyle{plain}

\newtheorem{thm}{Theorem}[section]

\newtheorem{cor}[thm]{Corollary}
\newtheorem{prop}[thm]{Proposition}
\newtheorem{lem}[thm]{Lemma}

\theoremstyle{definition}
\newtheorem{defn}[thm]{Definition}

\newtheorem{rmk}[thm]{Remark}

\def\GG{{\textbf G}}

\def\PP{{\textbf P}}
\def\SS{{\textbf S}}
\def\OO{\mathcal{O}}

\def\F{\mathcal{F}}
\def\P{\mathcal{P}}
\def\W{\mathcal{W}}
\def\V{\mathcal{V}}

\def\G{\mathcal{G}}

\def\I{\mathcal{I}}

\def\cM{\mathcal{M}}

\def\cZ{\mathcal{Z}}
\def\cU{\mathcal{U}}
\def\cV{\mathcal{V}}

\def\H{\mathcal{H}}

\author[M. Aprodu]{Marian Aprodu}
\address{Institute of Mathematics "Simion Stoilow" of the Romanian
Academy, Calea Grivi\c tei \indent 21, Sector 1, RO-010702 Bucharest, Romania}
 \email{{\tt
aprodu@imar.ro}}

\author[G. Farkas]{Gavril Farkas}
\address{Humboldt-Universit\"at zu Berlin, Institut F\"ur Mathematik,  Unter den Linden 6
\hfill \newline\texttt{}
 \indent 10099 Berlin, Germany} \email{{\tt farkas@math.hu-berlin.de}}

 \author[A. Ortega]{Angela Ortega}
\address{Humboldt-Universit\"at zu Berlin, Institut F\"ur Mathematik,  Unter den Linden 6
\hfill \newline\texttt{}
 \indent 10099 Berlin, Germany} \email{{\tt ortega@math.hu-berlin.de}}

\title{Minimal resolutions, Chow forms and Ulrich bundles on $K3$ surfaces}

\begin{document}
\maketitle

The Minimal Resolution Conjecture (MRC) for points on a projective variety $X\subset \PP^r$
predicts that the minimal graded free resolution of a general set $\Gamma\subset X$ of points is as simple as the geometry of $X$ allows. Originally, the most studied case has been that when $X=\PP^r$, see \cite{EPSW}.
The general form of the MRC for  subvarieties $X\subset \PP^r$
was formulated in \cite{Mus} and \cite{FMP}. The Betti diagram of a large enough set $\Gamma \subset X$ consisting of $\gamma$ general points is obtained from the Betti diagram of $X$, by adding two rows, indexed by $u-1$ and $u$, where $u$ is an integer depending on $\gamma$. All differences $b_{i+1, u-1}(\Gamma)-b_{i, u}(\Gamma)$ are known and depend on the Hilbert polynomial $P_X$ and $i, u$ and $\gamma$, see \cite{FMP}. The \emph{Minimal Resolution Conjecture}  for $\gamma$ general points on $X$ predicts that
$$b_{i+1, u-1}(\Gamma)\cdot b_{i, u}(\Gamma)=0,$$
for each $i\geq 0$, in which case, the Betti numbers of $\Gamma$ are explicitly given in terms of $P_X$ and $\gamma$. The \emph{Ideal Generation Conjecture} (IGC) predicts the same vanishing but only for $i=1$, that is, $b_{2, u-1}(\Gamma)\cdot b_{1, u}(\Gamma)=0$; equivalently, the number of generators of the ideal $I_{\Gamma}/I_X$ is minimal.
\vskip 3pt
In \cite{FMP}, the Minimal Resolution Conjecture for points on curves is reformulated in geometric terms. For a globally generated linear series  $\ell=(L, V)\in G^r_d(C)$, we consider the kernel vector bundle $M_V$ defined via the evaluation sequence
$$0\longrightarrow M_V\longrightarrow V\otimes \OO_C\longrightarrow L\longrightarrow 0.$$
Then MRC holds for $C\stackrel{|V|}\hookrightarrow \PP^r$ if and only if $M_V$ satisfies the \emph{Raynaud property} (R)
\begin{equation}\label{raynaud}
H^0\Bigl(C, \bigwedge ^i M_V\otimes \xi\Bigr)=0,
\end{equation}
for each $i=0, \ldots, r$ and a general line bundle $\xi$ on $C$ with $\mbox{deg}(\xi)=g-1+\lfloor \frac{id}{r}\rfloor$, see \cite{FMP} Corollary 1.8. When $\mu:=\frac{d}{r}\in \mathbb Z$ (in which case we refer to $C\subset \PP^r$ as being a curve of integer slope), property (R) is satisfied if and only if for $i=0, \ldots, r$, the cycle
$$\Theta_{\bigwedge^i M_V}:=\Bigl\{\xi\in \mathrm{Pic}^{g-1+i \mu}(C): h^0\Bigl(C, \bigwedge ^i M_V\otimes \xi\Bigr)\neq 0\Bigr\}$$ is a divisor in $\mathrm{Pic}^{g-1+i\mu}(C)$.  Equivalently, $\bigwedge^i M_V$ has a theta divisor for all $i\geq 0$.

Our first result is a proof of MRC for curves $C\subset \PP^r$ of integer slope $\mu:=\frac{d}{r}\in \mathbb Z_{\geq 1}$.

\begin{thm}\label{MRC0}
The Minimal Resolution Conjecture holds for a general embedding $C\hookrightarrow \PP^r$
of degree $\mu r$ of any curve $C$ with general moduli and for any integers $\mu, r\geq 1$.
\end{thm}

The hypothesis on the generality of $C$ implies that its genus $g$ satisfies the inequality $g\leq (r+1)(\mu-1)$ imposed by Brill-Noether theory. We have similarly complete results for curves $C\subset \PP^r$ of degree $d\equiv \pm 1 \mbox{ mod } r$, see Theorem \ref{pm}.

In the case of curves $C\stackrel{|L|}\hookrightarrow \PP^{d-g}$ embedded by a complete linear system of degree $d\geq 2g+5$, counterexamples to MRC for points on $C$ were found in \cite{FMP}\footnote{Note that it is precisely in this range the Minimal Resolution Conjecture fails to hold for $C$ itself as well; if $L\in \mathrm{Pic}^{d}(C)$ is a line bundle of degree $d\geq 2g+5$, then the resolution of $C\stackrel{|L|}\hookrightarrow \PP^{d-g}$ is not natural in the sense of \cite{CEFS}, that is, there exists $i$ such that $b_{i+1,1}(C)\cdot b_{i,2}(C)\neq 0$.}; observe that in these cases $\mu=\frac{d}{d-g}<2$. On the other hand, MRC holds for \emph{all} smooth canonical curves $C\subset \PP^{g-1}$, see \cite{FMP}, as well as for general line bundles of degree $2g$, see \cite{B1}. In both these cases, one has $\mu=2$. This confusing state of affairs is reminiscent of the situation for the projective space $\PP^r$, where it is known \cite{HS} that MRC holds for $r\geq 4$ and $\gamma$ very large with respect to $r$, but fails for each $r\geq 6, r\neq 9$ for many values of $\gamma$, see \cite{EPSW}. Our next result show that for curves, independently of the genus, the \emph{Clifford line} line $d=2r$ in the $(d, r)$-plane governs whether MRC holds for a general curve $C\subset \PP^r$ of genus $g$ and degree $d$.

\begin{thm}\label{MRC1}
Let $C$ be a curve of genus $g$ with general moduli and integers $d, r\geq 1$ such that $d\geq 2r$. The Minimal Resolution Conjecture holds for a general embedding
 $C\hookrightarrow \PP^r$ of degree $d$, whenever the following condition is satisfied:
 \begin{equation}\label{felt}
  d+r \Bigl\lfloor \frac{d}{r} \Bigr\rfloor \geq 2g+2r-2.
\end{equation}
\end{thm}
Note that no assumption is made regarding the completeness of the linear series $(L, V)$ inducing the map $\varphi_{V}: C \hookrightarrow \PP^r$. Inequality (\ref{felt}) in Theorem \ref{MRC1} is satisfied when $d\geq g+\frac{3r}{2}-2$. It is also satisfied in the range $d\geq 2g-1$, when all line bundles in question are non-special. The condition $d\geq 2r$ is certainly necessary, for as already pointed out, in the other cases counterexamples to MRC were produced using complete linear series, see \cite{FMP} Theorem 2.2.
\vskip 4pt

We now turn our attention to the IGC for a set $\Gamma$ of $\gamma$ general points lying on an embedded curve $\varphi_V: C\hookrightarrow \PP^r$.
Assume $\gamma\geq d\cdot \mbox{reg}(C)-g+1$ and set $u:=1+\lfloor \frac{\gamma+g-1}{d}\rfloor$; thus $u$ is the integer uniquely determined by the condition
$P_C(u-1)\leq \gamma<P_C(u)$,
see also Section 1 for details. The resolution of the zero-dimensional scheme $\Gamma\subset \PP^r$ has the following form, see also \cite{Mus} Proposition 1.6,
$$\cdots\rightarrow S(-u)^{\oplus (du+1-g-\gamma)}\oplus S(-u-1)^{\oplus b_{1, u}(\Gamma)}\rightarrow S\rightarrow S(\Gamma)\rightarrow 0,$$
where $b_{2, u-1}(\Gamma)-b_{1, u}(\Gamma)=r(du-\gamma+1-g)-d$. The Ideal Generation Conjecture  for $C$ and $\Gamma$ amounts to the multiplication map
$$V\otimes H^0\bigl(C, \I_{\Gamma/C}(u)\bigr)\rightarrow H^0\bigl(C, \I_{\Gamma/C}(u+1)\bigr)$$ having maximal rank, or equivalently, the number of generators of the ideal $I_{\Gamma}/I_C$ being minimal, precisely $b_{1, u}(\Gamma)=\mbox{max}\bigl\{d-r(du-\gamma+1-g), 0\bigr\}$. The following result gives a complete solution to IGC for general curves.

\begin{thm}\label{igc}
Fix integers $g, r, d\geq 0$. Then the Ideal Generation Conjecture holds for a general embedding $C\hookrightarrow \PP^r$ of degree $d$ of any genus $g$ curve $C$ having general moduli.
\end{thm}

It should be pointed out that Theorems \ref{MRC0}, \ref{MRC1} and \ref{igc} are optimal in the sense that they establish MRC or IGC for a \emph{general} curve $[C]\in \cM_g$ and a \emph{general} linear series $\ell\in G^r_d(C)$. Having fixed $g, r$ and $d$, one cannot expect a more precise statement. It suffices indeed to consider the situation in genus zero. To a non-degenerate rational curve $R\subset \PP^r$ of degree $d$, one associates the splitting type $a_1\leq \ldots \leq a_r$ of the vector bundle $T_{\PP^r|R}(-1)=\OO_{\PP^1}(a_1)\oplus \cdots \oplus \OO_{\PP^1}(a_r)$. The splitting type of a general $R$ as above is \emph{balanced}, that is, with $0\leq a_r-a_1\leq 1$, and then $a_1=\lfloor \frac{d}{r}\rfloor$ and $a_r=\lceil \frac{d}{r} \rceil$; the locus of curves with non-balanced splitting type is a divisor in the (irreducible) Hilbert scheme of rational curves $R\subset \PP^r$ of degree $d$. On the other hand, it is easy to see cf. \cite{Mus} Corollary 3.8, that $R$ verifies MRC if and only if its splitting type is balanced. Such examples can be constructed on every curve of positive genus, by considering linear series with exceptional secant behaviour, which renders the restriction of $T_{\PP^r}(-1)$ too unstable. It turns out, that systematically MRC fails along certain proper subvarieties of the Hilbert scheme in question, but holds generically.

\vskip 5pt

The second topic we investigate in this paper concerns Chow forms and Ulrich bundles. We fix a $k$-dimensional variety $X\subset \PP^r$ of degree $d$. Following \cite{ES}, a vector bundle $E$ on $X$ is said to be an \emph{Ulrich bundle} if $E$ admits a \emph{completely linear} $\OO_{\PP^r}$-resolution
$$0\rightarrow \OO_{\PP^r}(-r+k)^{\oplus a_{r-k}}\rightarrow \cdots \rightarrow \OO_{\PP^r}(-1)^{\oplus a_1}\rightarrow \OO_{\PP^r}^{\oplus a_0}\rightarrow E\rightarrow 0,$$
where $a_0= d\cdot \mbox{rk}(E)$ and $a_i={r-k\choose i}a_0$ for $i\geq 1$. In terms more intrinsic to $X$, this amounts to requiring $E$ to be an ACM bundle, that is, $H^i(X, E(t))=0$ for all
$t$ and $i=1, \ldots, k-1$, and the module ${\bf{\Gamma}}_*(E):=\oplus_{t\in \mathbb Z} H^0(X, E(t))$  to have the maximum number of generators, which equals $d\cdot \mbox{rk}(E)$, all appearing in degree $0$. It is conjectured in \cite{ES} that every $k$-dimensional projective variety $X\subset \PP^r$ carries an Ulrich bundle. As explained in \cite{ES2}, the existence of an Ulrich bundle on $X$ implies that the \emph{cone of cohomology tables}
$$C(X, \OO_X(1)):=\mathbb Q_{\geq 0} \Bigl\langle \Bigl(h^i(X, F(m))\Bigr)_{0\leq i\leq k, m\in \mathbb Z}: F \in \mbox{Coh}(X) \Bigr \rangle \subset \mbox{Mat}_{k+1, \infty}(\mathbb Q)$$
is the same as that for the projective space $\PP^k$.
This conjecture has been confirmed so far only in few cases. A hypersurface carries an Ulrich bundle of exponential rank, see \cite{BHU}. Curves also carry Ulrich line bundles \cite{ES};  a vector bundle $E$ on a smooth curve $C\subset \PP^r$ having slope $\mu(E)=d+g-1$ is an Ulrich bundle, if and only
$$H^0(C, E(-1))=0\Leftrightarrow \OO_C(-1) \notin \Theta_E.$$

\vskip 3pt

When $X\subset \PP^r$ is a hypersurface, the existence of Ulrich bundles is related to classical problems in algebraic geometry, see \cite{B2}. If $\mbox{rk}(E)=1$, then one has a determinantal presentation of $X:\{\mbox{det}(M)=0\}$, where $M=(\ell_{ij})_{1\leq i, j\leq d}$ is a matrix of linear forms; a bundle $E$ with $\mbox{rk}(E)=2$ corresponds to a Pfaffian equation of $X: \{\mbox{pf}(M)=0\}$, where $M$ is a $(2d)\times (2d)$ skew-symmetric linear  matrix.

\vskip 3pt

Del Pezzo surfaces $X_d\subset \PP^d$ of degree $d$  have Ulrich bundles of any rank $r\geq 2$, see \cite{MP}, \cite{CHGS}. A remarkable connection between Ulrich bundles and the Ideal Generation Conjecture as studied in this paper is established in \cite{CKM1}. Precisely, there exists an Ulrich bundle $E$ on $X$ with $\mbox{det}(E)=\OO_X(C)$, if and only if the curve $C\subset X$ has degree $d\cdot \mbox{rk}(E)$ and IRC holds for $C$. Finally, we mention that using the techniques of \cite{AF11}, Coskun, Kulkarni and Mustopa \cite{CKM12} have shown that every smooth quartic surface $X\subset \PP^3$ carries a rank $2$ Ulrich bundle, thus generalizing work of Beauville \cite{B2}.

\vskip 5pt

In this paper we describe the moduli space of Ulrich bundles on a polarized $K3$ surface. First, we show that $K3$ surfaces satisfying a mild generality condition carry Ulrich bundles of rank $2$, satisfying the skew-symmetry requirement of \cite{ES}.

\begin{thm}
\label{thm: ulrich}
Let $S\subset \PP^{s+1}$ be a polarized $K3$ surface. If the Clifford index of cubic sections of $S$ is computed by $\mathcal O_S(1)$, then $S$ carries a $(2s+10)$-dimensional family of stable rank $2$ Ulrich bundles $E$ with $\mathrm{det}(E)=\mathcal O_S(3)$.
\end{thm}

A smooth cubic section $C\in |\OO_S(3)|$ is a curve of genus $9s+1$ and degree $6s$, hence the inequality $\mbox{Cliff}(C)\leq \mbox{Cliff}(\OO_C(1))=4s-2$  holds.
If $\mbox{Cliff}(C)<4s-2$, then from \cite{GL} it follows that there exists an effective class $D\in \mbox{Pic}(S)$ such that
\begin{equation}\label{nöl}
h^0(S, \OO_S(D))\geq 2, \ \ D\cdot H\leq 3s \ \ \mbox{ and    } \ 3D\cdot H-D^2\leq 4s-1.
\end{equation}

The existence of rank $2$ Ulrich bundles is established for all $K3$ surfaces in the complement of the Noether-Lefschetz locus singled out by (\ref{nöl}). In particular, Theorem \ref{thm: ulrich} holds for every $K3$ surface of Picard number one.  Using \cite{L2} page 185, condition (\ref{nöl}) is never satisfied for smooth quartics (that is, $s=2$), for in this case we have complete intersection curves, whose Clifford index is computed by multisecants.  The restriction to rank $2$ is natural, for a very general polarized $K3$ surface carries \emph{no} Ulrich bundles of odd rank, see Section 2.
\vskip 3pt

The case $s=2$ of Theorem \ref{thm: ulrich} was proved in \cite{CKM12}. Our proof of Theorem \ref{thm: ulrich} partly grew out of an attempt to generalize that result. The bundles $E$ are \emph{special Ulrich bundles} in the sense of \cite{ES} Proposition 6.2; when $\mbox{det}(E)=\OO_S(3)$ the Ulrich condition is equivalent to $E$ being $0$-regular. The candidate bundles are \emph{Lazarsfeld-Mukai bundles} $E:=E_{C, A}$, where $C\in |\OO_S(3)|$ is a suitable cubic section of $S$ and $A\in W^1_{5s+4}(C)$ is a complete base point free pencil. Since $C$ is far from being Brill-Noether general, showing that a general cubic section $C\subset S$ carries a pencil $\mathfrak g^1_{5s+4}$ inducing a simple Ulrich bundle, becomes a rather tricky variational problem, which we solve in a way reminiscent of our proof \cite{AF11} of Green's conjecture for curves on arbitrary $K3$ surfaces. The role of the required equality $\mbox{Cliff}(C)=\mbox{Cliff}(\OO_C(1))$ is that it ensures the existence of a complete  base point free pencil $\mathfrak g^1_{5s+4}$.
\vskip 3pt

By taking direct sums of Ulrich bundles, Theorem \ref{thm: ulrich} implies the existence of Ulrich bundles of any even rank on $S$. We show that for very general $K3$ surfaces these direct sums can be deformed to stable Ulrich bundles:
\begin{thm}
Let $S\subset \PP^{s+1}$ be a polarized $K3$ surface with $\mathrm{Pic}(S)=\mathbb Z\cdot [H]$. For every integer $a\geq 1$, there exists an $(8a^2+2a^2s+2)$-dimensional family of \emph{stable} Ulrich bundles of rank $2a$. Furthermore, there exists a component of the corresponding moduli space $M_S\bigl(2a, \OO_S(3a), 9a^2s-4a(s-1)\bigr)$ of the moduli space of vector bundles on $S$, whose general point corresponds to a stable Ulrich bundle of rank $2a$.
\end{thm}
The case $s=2$ of this result has been recently established by Coskun \cite{C} and our method of proof is very similar to his. Since $K3$ surfaces with Picard number one carry no Ulrich bundles of odd rank, this answers completely the question which numbers appear as ranks of stable Ulrich bundles on $S$. In the language of \cite{MP}, Theorem \ref{highrank} establishes that every $K3$ surface of Picard number one has \emph{wild representation type}.

\vskip 4pt

Our results on Ulrich bundles imply via \cite{ES} that the Chow form of a polarized $K3$ surface admits a pfaffian B\'ezout form in Pl\"ucker coordinates. We fix as before a polarized $K3$ surface
$S\subset \PP^{s+1}$ and let $\GG:=G(s-1, H^0(S, \OO_S(1))^{\vee})$ be the Grassmannian of codimension $3$ planes in $\PP^{s+1}$ and $\cU$ the rank $3$ tautological bundle on $\GG$ sitting in the exact sequence
$$0\longrightarrow \cU \longrightarrow H^0(S, \OO_S(1))\otimes \OO_{\GG} \longrightarrow \mathcal Q \longrightarrow 0.$$
Let ${\bf \Lambda}:=\bigwedge^{\bullet} H^0(S, \OO_S(1))^{\vee}$ be the exterior algebra and ${\bf \Lambda}^{\vee}:=\bigwedge^{\bullet} H^0(S, \OO_S(1))$ its dual.
Using the identification $\bigwedge^{s-1} H^0\bigl(S, \OO_S(1)\bigr)^{\vee}\cong \bigwedge ^3 H^0\bigl(S, \OO_S(1)\bigr)$, we view elements from $\bigwedge^3 H^0\bigl(S, \OO_S(1)\bigr)^{\vee}$ as Pl\"ucker coordinates on $\GG$. We recall that the \emph{Cayley-Chow} form of $S$ is the degree $2s$ hypersurface
$$\cZ(S):=\{L\in \GG: L\cap S\neq \emptyset\}.$$
Putting together Theorem \ref{thm: ulrich} and \cite{ES} Corollary 3.4, we obtain the following result:

\begin{thm}\label{chow}
Let $S\subset \PP^{s+1}$ be a polarized $K3$ surface such that the Clifford index of a cubic section is computed by $\OO_S(1)$ and $E$ a rank two Ulrich bundle on $S$. Then there exists a skew-symmetric morphism of vector bundles of rank $4s$ on $\GG$
$$\mathrm{{\bf{U}}}_3(\varphi):H^0(S, E)^{\vee}\otimes \bigwedge ^3 \cU\longrightarrow H^0(S, E)\otimes \OO_{\GG}$$
whose pfaffian cuts out precisely the Cayley-Chow form of $S$.
\end{thm}

Via the identification $H^2(S, E(-3))\cong H^0(S, E)^{\vee}$, the morphism ${\bf{U}_3}(\varphi)$ is given by applying the functor ${\bf{U}}_3$ from \cite{ES} to the map
$$\varphi: H^2(S, E(-3))\otimes \bigwedge^3 H^0(S, \OO_S(1))\rightarrow H^0(S, E).$$
 Precisely, there exists an exact sequence of ${\bf \Lambda}$-modules
\begin{equation}\label{tate}
 {\bf \Lambda}^{\vee}(3)\otimes H^2(S, E(-3)) \stackrel{\varphi}\longrightarrow {\bf \Lambda}^{\vee} \otimes H^0(S, E) \stackrel{\psi}\longrightarrow {\bf \Lambda}^{\vee}(-1)\otimes H^0(S, E(1)),
 \end{equation}
where $\psi=(\psi^j)_{j\geq 0}$ is the map given by the Koszul differentials
$$\psi^j:\bigwedge^j H^0(S, \OO_S(1))\otimes H^0(S, E)\longrightarrow \bigwedge^{j-1} H^0(S, \OO_S(1))\otimes H^0(S, E(1)).$$
As a map of ${\bf \Lambda}$-modules, $\psi$ is given by the same tensor as the multiplication map
$$H^0(S, \OO_S(1))\otimes H^0(S, E)\rightarrow H^0(S, E(1)).$$

The complex (\ref{tate}) is part of the \emph{Tate resolution} of the module ${\bf{\Gamma}}_*(E)$ over the exterior algebra ${\bf \Lambda}$, as constructed in \cite{EFS}. The vector bundle morphism in Theorem \ref{chow} is obtained by applying the functor of Eisenbud and Schreyer \cite{ES} to the Tate resolution. This  reduces the complex to a single morphism of vector bundles on $\GG$, which degenerates exactly along $\cZ$.
For an explicit description of $\varphi$, we refer to Section 2. In the last section of the paper we prove a Minimal Resolution Conjecture type result for zeros of sections in the twists of Ulrich bundles on $S$ (see Proposition \ref{ulmrc}) and we compare it against the Minimal Resolution Conjecture.

\vskip 6pt

\vskip 3pt
\noindent {\bf Acknowledgments:} The first author thanks the Max Planck Institut f\"ur Mathematik Bonn and the Humboldt Universit\"at zu Berlin for hospitality during the preparation of this work. The second author is grateful to Frank Schreyer for very useful discussions on matters related to this circle of ideas.
The first author was partly supported by the CNCS-UEFISCDI grant  PN-II-PCE-2011-3-0288 and by a Humboldt fellowship. The second and third authors were partly supported by the SFB 647 "Raum-Zeit-Materie". We thank the referee for a number of pertinent comments that clearly improved the presentation of this paper.

\section{Minimal resolutions of sets of points on curves and theta divisors}

The aim of this section is to prove Theorems \ref{MRC0} and \ref{MRC1} and we begin with preliminaries, see  also \cite{G}, \cite{FMP}. The \emph{graded Betti numbers} of a subscheme $Z\subset \PP^r$, counting the $i$-th order syzygies of degree $j$ in the minimal free resolution of the coordinate ring $S(Z)$ over the polynomial ring $\SS:=\mathbb C[x_0, \ldots, x_r]$, are denoted as usual by $$b_{i, j}(Z):=\mbox{dim}_{\mathbb C}\ \mbox{Tor}_{i}^{\SS} \bigl(S(Z), \mathbb C)_{i+j}=\mbox{dim}_{\mathbb C}\  K_{i, j}(Z, \OO_Z(1)).$$
The graded Betti diagram of $Z$ is obtained by placing $b_{i, j}(Z)$ in the $j$-th row and $i$-th column. The number of non-trivial rows in the Betti diagram of $Z$ equals the Castelnuovo-Mumford  regularity $\mbox{reg}(Z)$, that is, $b_{i, j}(Z)=0$, for $j\geq \mbox{reg}(Z)+1$.
\vskip 3pt

Let $C\subset \PP^r$ be a smooth curve of genus $g$ embedded by a not necessarily complete linear series $\ell:=(L, V)\in G^r_d(C)$. The kernel bundle  $M_V:=\mbox{Ker}\{V\otimes \OO_C\rightarrow L\}$ of the evaluation map can be interpreted as $M_V=\Omega^1_{\PP^r |C}(1)$. We fix a set $\Gamma\subset C$ of $\gamma$ general points, where $\gamma\geq d\cdot \mbox{reg}(C)+1-g$, then set
$$u:=1+\Bigl \lfloor \frac{\gamma+g-1}{d} \Bigr \rfloor\geq 1+\mathrm{reg}(C).$$
It is proved in \cite{FMP} Theorem 1.2 that the Betti diagram of $\Gamma$ is obtained from that of $C$ by adding two rows, indexed by $u-1$ and $u$ respectively. Precisely, one has that
$$b_{i, j}(\Gamma)=b_{i, j}(C), \ \mbox{ for } i\geq 0, \ j\leq u-2, \ \ \ \mbox{ and }$$
$$b_{i, j}(\Gamma)=0, \mbox{ for } \ i\geq 0 \mbox{ and } j\geq u+1.$$
The Betti numbers of $\Gamma$ in rows $u-1$ and $u$ have the following interpretation:
$$b_{i+1, u-1}(\Gamma)=h^0\Bigl(C, \bigwedge^i M_V\otimes L^{\otimes u}(-\Gamma)\Bigr) \ \mbox{   and }\  b_{i, u}(\Gamma)=h^1\Bigl(C, \bigwedge ^i M_V\otimes L^{\otimes u}(-\Gamma)\Bigr).$$
The difference of the two Betti numbers on each diagonal can be computed via Riemann-Roch, being equal to the Euler characteristic of a vector bundle on $C$:
$$b_{i+1, u-1}(\Gamma)-b_{i, u}(\Gamma)=\chi \Bigl(C, \bigwedge^i M_V\otimes L^{\otimes u}(-\Gamma)\Bigr)={r\choose i} \Bigl(-\frac{id}{r}+du-\gamma+1-g\Bigr).$$
\vskip 4pt

The \emph{Minimal Resolution Conjecture} (MRC) for $C$ predicts that $b_{i+1, u-1}(\Gamma)\cdot b_{i, u}(\Gamma)=0$ for all $i$, that is, the number of syzygies of $\Gamma$ is as small as the parameters $g, d, r, u$ and $\gamma$ allow. The \emph{Ideal Generation Conjecture} (IGC) predicts the same vanishing, but only for $i=1$.  The MRC (respectively IGC) for $C$ break up into generic vanishing statements for exterior powers of kernel bundles.

\begin{prop}\label{mrcc} (a) The Minimal Resolution Conjecture holds for a smooth curve $C\subset \PP^r$, if and only if $H^0\bigl(C, \bigwedge ^i M_V\otimes \xi\bigr)=0$ for all $i=1, \ldots, r-1$ and a general line bundle  $\xi\in \mathrm{Pic}(C)$ with $\mathrm{deg}(\xi)=g-1+\lfloor \frac{id}{r}\rfloor$.

\noindent (b) The Ideal Resolution Conjecture holds for $C\subset \PP^r$, if and only if the previous generic vanishing statement holds for $i=1, r-1$.
\end{prop}

As already observed in \cite{FMP}, the vanishing statements in Proposition \ref{mrcc} are closely related to work of Raynaud \cite{R}.

\begin{defn} Let $C$ be a smooth curve of genus $g$ and $E$ a vector bundle on $C$ with slope $\mu(E)=\mu$. Then $E$  is said to satisfy condition (R), if $H^0\bigl(C, \bigwedge^i E\otimes \xi\bigr)=0$, for all $i=1, \ldots, r-1$ and for a  general line bundle $\xi\in \mathrm{Pic}^{g-1-\lceil i\mu\rceil}(C)$.
\end{defn}

When $\mu \in \mathbb Z$, condition (R) implies the semistability of the vector bundle $E$ and it is in general a much stronger property.
Raynaud \cite{R} has given the first examples of stable vector bundles on curves of genus at least $4$ that do not satisfy condition (R). Popa \cite{P} showed that if $\mbox{deg}(L)\geq 2g+2$, then the kernel bundle $M_L$  fails to verify condition (R). When $\mu(E)=\mu\in \mathbb Z$, the bundle $E$ verifies condition (R) if and only if $\bigwedge ^i E$ admits a theta divisor $\Theta_{\bigwedge ^i E}\subset \mbox{Pic}^{g-i\mu-1}(C)$ for all $i$.

\vskip 4pt

Let us fix integers $g, r, d\geq 1$, such that the Brill-Noether number
$$\rho(g, r, d):=g-(r+1)(g-d+r)$$ is non-negative. The Hilbert scheme $\mbox{Hilb}_{g, r, d}$ of curves $C\subset \PP^r$ of genus $g$ and degree $d$ has a unique component $\H_{g, r, d}$ with general point corresponding to a smooth curve and which maps dominantly onto $\cM_g$ under the forgetful map $\sigma:\H_{g, r, d}\dashrightarrow \cM_g$. In order to prove MRC for a general embedding of a curve of genus $g$ with general moduli, it suffices, for given $r$ and  $d$, to construct a smooth embedded curve $\bigl[C\stackrel{|V|}\longrightarrow \PP^r\bigr]$ which (i) lies in the component $\H_{g, r, d}$, and for which (ii) the bundle $M_V$ verifies the property (R). Condition (i) is implied by the injectivity of the Petri map $$\mu_0(V):V\otimes H^0(C, K_C(-1))\rightarrow H^0(C, K_C),$$ which is automatically satisfied in the non-special range $d\geq 2g-1$.
\vskip 4pt

We now prove Theorem \ref{MRC0} for curves of integral slope $\mu=\frac{d}{r}\in \mathbb Z$. For an integer $\mu\geq 1$, the inequality $\rho(g, r, \mu r)\geq 0$ is equivalent to $g\leq (r+1)(\mu-1)$. If $C\subset \PP^r$ is a nodal curve, when there is no danger of confusion, we write $M_C:=\Omega_{\PP^r|C}^1(1)=M_V$, where $V\subset H^0(C, \OO_C(1))$ is the space of sections inducing the embedding of $C$.

\vskip 5pt

\noindent \emph{Proof of Theorem \ref{MRC0}.}
When $\mu=1$, then $C\subset \PP^r$ is necessarily a rational normal curve and $M_C=\OO_{\PP^1}(-1)^{\oplus r}$. The conclusion of the theorem is immediate.

Suppose now that $\mu\geq 2$. We specialize to a $\mu$-gonal curve of genus $g$ in such a way that the corresponding kernel bundle splits into a direct sum of line bundles of the same degree. Let $[C]\in \cM_{g, \mu}^1$ be a general member of the $\mu$-gonal locus in $\cM_g$. Then the \emph{scrollar invariants} of a suitably general pencil $\mathfrak g^1_{\mu}$ on $C$ are as balanced as possible. Precisely, $C$ possesses a base point free pencil $(A, W)\in G^1_{\mu}(C)$, such that
$H^0(C, A^{\otimes j})=j+1$ if and only if $g\geq j(\mu-1)$; else, that is, when $g\leq j(\mu-1)$, we have that
$H^1(C, A^{\otimes j})=0$. In particular, the assumption $\rho(g, r, \mu r)\geq 0$ implies that $H^1(C, A^{\otimes (r+1)})=0$, see \cite{CM} Proposition 2.1.1. For the rest of the proof we fix such a pencil $(A,W)\in G^1_{\mu}(C)$, where $W=H^0(C,A)$, and consider the following triple
$$\Bigl[C,\ L:=A^{\otimes r},\  V:=\mathrm{Sym}^r(W)\Bigr]\in \mathrm{Hilb}_{g, r, \mu r},$$
where we identify $\mathrm{Sym}^r(W)$ with its image under the injection
$\mbox{Sym}^r(W)\hookrightarrow H^0(C, A^{\otimes r})$. This point corresponds to a complete linear series, that is, $V=H^0(C, A^{\otimes r})$ if and only $g\in [r(\mu-1), (r+1)(\mu-1)]$, or after setting $d:=\mu r$, equivalently when $g-d+r\geq 0$. Geometrically, the constructed curve is given by the map $\nu_r\circ \varphi:C\rightarrow \PP^r$, where $\varphi:C\rightarrow \PP^1$ is the degree $\mu$ map corresponding to the pencil $|W|$ and $\nu_r:\PP^1\rightarrow \PP^r$ is the $r$th Veronese map, whose image is a rational normal curve $R\subset \PP^r$.

\vskip 3pt
The kernel bundle $M_R=\Omega^1_{\PP^r |R}(1)$ splits into a sum of line bundles of the same degree, precisely,
$M_R=\OO_{\PP^1}(-1)^{\oplus r}$. Moreover, $M_V=\varphi^*(M_R)=(A^{\vee})^{\oplus r}$, hence
 $$\bigwedge^i M_V=\Bigl(A^{\otimes (-i)}\Bigr)^{\oplus {r\choose i}},$$
 for $i=1, \ldots, r-1$. Since a direct sum of line bundles of the same degree has a theta divisor, we are left with proving that $[C, L, V]$ belongs to the main component $\H_{g, r, \mu r}$ of the Hilbert scheme. It suffices to show that the Petri map
 $$\mu_0(V):\mbox{Sym}^r(W)\otimes H^0(C, K_C\otimes A^{\otimes (-r)})\rightarrow H^0(C, K_C)$$ is injective.
 This is automatic when $g\leq r(\mu-1)$, because then $H^1(C, A^{\otimes r})=0$.

 \vskip 3pt
 We consider the case where
 $r(\mu-1)\leq g\leq (r+1)(\mu-1)$, when the linear series $A$ is complete, $h^0(C, A^{\otimes r})=r+1$ and the map $\nu_r \circ \varphi$ corresponds to a complete linear series.

 \vskip 4pt

 We prove by induction that for each $1\leq j\leq r$, the multiplication map
 $$\chi_j:\mbox{Sym}^j(W)\otimes H^0(C, K_C\otimes A^{\otimes (-r)})\rightarrow H^0(C, K_C)$$
 is injective. Note that $\chi_r=\mu_0(V)$ is just the Petri map, which will conclude the proof.
 Suppose $\chi_{j-1}$ is known to be injective and assume that $\mbox{Ker}(\chi_r)\neq 0$.
 After choosing a basis $(s_1, s_2)$ for $W$, we find sections
 $u_1, \ldots, u_{j+1}\in H^0(K_C\otimes A^{\otimes (-r)})$ such that
 \begin{equation}\label{bpfpt}
 s_1^j\cdot u_1+(s_1^{j-1}s_2)\cdot u_2+\cdots+(s_1s_2^{j-1})\cdot u_j=s_2^{j}\cdot u_{j+1}.
 \end{equation}

 Then $u_{j+1}\neq 0$, for else, $\sum_{k=1}^{j} (s_1^{j-k} s_2^{k-1})\otimes u_k\in \mbox{Sym}^{j-1}(W)\otimes H^0\bigl(K_C\otimes (A^{\vee})^{\otimes (-r)}\bigr)$ is a non-zero element in the kernel of $\chi_{j-1}$, a contradiction. Applying the Base Point Free Pencil Trick to equality (\ref{bpfpt}), we obtain a non-zero section $x_1\in H^0\bigl(K_C\otimes (A^{\vee})^{\otimes (r-j+2)}\bigr)$ such that the following equalities hold in $H^0\bigl(C, K_C\otimes (A^{\vee})^{\otimes (r-j+1)}\bigr)$:
 $$s_1\cdot x_1=s_2^{j-1}\cdot u_{j+1} \ \ \mbox{ and } \ s_1^{j-1}\cdot u_1+\cdots+s_2^{j-1}\cdot u_j=-s_2\cdot x_1.$$
 Applying again the Base Point Free Pencil Trick to the first of these equalities, we find a section $0\neq x_2\in H^0\bigl(C, K_C\otimes (A^{\vee})^{\otimes (r-j+3)}\bigr)$, such that
$$x_1=-s_2\cdot x_2 \ \mbox{ and } \ s_2^{j-2}\cdot u_{j+1}=s_1\cdot x_2.$$

Repeating this argument, we obtain a section $0\neq x_{j-1}\in H^0\bigl(C, K_C\otimes (A^{\vee})^{\otimes (-r)}\bigr)$, such that $s_2\cdot u_{j+1}=s_1\cdot x_{j-1}$. So, we can write
$$s_1\otimes x_{j-1}-s_2\otimes u_{j+1}\in \mbox{Ker}(\chi_1)\cong H^0(C, K_C\otimes (A^{\vee})^{\otimes (-r-1)})=0.$$ Therefore $u_{j+1}=0$. This is a contradiction, hence $\nu_r\circ \varphi:C\rightarrow \PP^r$ lies in $\H_{g, r, \mu r}$.
\hfill $\Box$
\vskip 3pt

The following result must be well-known and it follows easily from Atiyah's classification of vector bundles on elliptic curves.

\begin{prop}\label{atiyah}
Let $E$ be an elliptic curve, $B\in \mathrm{Pic}^b(E)$ a line bundle of degree $b\geq 2$ and an integer $1\leq r\leq b-1$. Then the kernel bundle $M_{V}$ corresponding to a general $(r+1)$-dimensional subspace $V\subset H^0(E, B)$ is semistable.
\end{prop}
\begin{proof}
We fix a semistable vector bundle $F$ on $E$ of rank $r$ with $\mbox{det}(F)=B$. Note that $\mu(F)=\frac{b}{r}>1$. For every point $p\in E$, one has $\mu(F(-p))=\mu(F)-1>0$, therefore
$H^1(E, F(-p))=0$. In particular, $F$ is globally generated. By Riemann-Roch, $h^0(E, F)=b\geq r+1$. A globally generated vector bundle $F$ on a curve is generated by a general set of $(\mbox{rk}(F)+1)$ sections. We choose a generating subspace $W\subset H^0(C, F)$ with $\mbox{dim}(W)=r+1$ and write the exact sequence
$$0\longrightarrow B^{\vee} \longrightarrow W\otimes \OO_E \longrightarrow F\longrightarrow 0.$$ By dualizing, we take $V:=W^{\vee}\subset H^0(E, B)$ and then $M_V=F^{\vee}$ is semistable.
\end{proof}

\vskip 4pt
Next we use a specialization to the bielliptic locus in $\cM_g$ that will be of use in the proof of Theorem \ref{MRC1} for curves not of integral slope.

\vskip 2pt

\begin{prop}\label{bielliptic}
Let $f:C\rightarrow E$ be a bielliptic curve of genus $g$ and $(B, V)\in G^r_b(E)$ a general linear series, where $r+1\leq b$. Then the kernel bundle corresponding to the pair $\ell:=(f^*(B), f^*(V))\in G^r_{2b}(C)$ verifies condition (R). Moreover,  for $b\geq g-2$, the Petri map corresponding to $\ell$ is injective, hence $\ell\in \H_{g, r, 2b}$.
\end{prop}
\begin{proof}
From Proposition \ref{atiyah} it follows that we can choose the pair $(B, V)$ such that $M_{V}$ is semistable. The cover $f:C\rightarrow E$ is characterized by a line bundle $\delta\in \mbox{Pic}^{g-1}(E)$ with
$$f_*(\OO_C)=\OO_E\oplus \delta^{\vee} \ \mbox{ and } \ \delta^{\otimes 2}=\OO_E(\mathfrak b),$$
where $\mathfrak b\in E_{2g-2}$ is the branch divisor of $f$. By pulling-back to $C$ the exact sequence
$$0\longrightarrow M_{V, B}\longrightarrow V\otimes \OO_E\longrightarrow B\longrightarrow 0,$$
we find that $M_{f^*(V), f^*(B)}=f^*(M_{V, B})$. Since $K_C=f^*(\delta)$, via the push-pull formula we obtain
$H^0(C, K_C\otimes f^*(B^{\vee}))=f^* H^0(E, \delta\otimes B^{\vee})$; the Petri map corresponding to $\ell$ is
essentially the multiplication map $V\otimes H^0(E, \delta\otimes B^{\vee})\rightarrow H^0(E, \delta)$. This is injective when
$h^0(E, \delta\otimes B^{\vee})\leq 1$, that is, $b\geq g-2$ (Note that $f^*(B)$ is non-special for $b\geq g-1$).
\vskip 3pt
It remains to check that $M_{f^*(V)}$ verifies property (R). Pick an integer $1\leq i \leq r-1$ and a general line bundle
$\xi \in \mbox{Pic}^{g-1+\lfloor \frac{id}{r}\rfloor}(C)$. From the formula $\mbox{det}(f_*\xi)=\mbox{Nm}_f(\xi)\otimes \delta^{\vee}$, coupled with Lemma 2.5 from \cite{CEFS}, it follows that $f_*\xi$ is a general \emph{semistable} vector bundle on $E$ of rank $2$ and degree $\lfloor \frac{id}{r}\rfloor$.
Then because of the semistability of the exterior powers of $M_V$ we obtain that
$$H^0\Bigl(C, \bigwedge ^i M_{f^*(V)}\otimes \xi\Bigr)=H^0\Bigl(E, \bigwedge^i M_V \otimes f_*\xi\Bigr)=0,$$
for $\bigwedge^i M_V\otimes f_*\xi$ is a general semistable vector bundle of slope $\frac{1}{2}\lfloor \frac{2ib}{r}\rfloor-\frac{ib}{r}\leq 0.$

\end{proof}

\subsection{Smoothing techniques.} The proof of Theorems \ref{MRC1} and \ref{igc} is by induction on the degree and genus. The inductive step uses the smoothing techniques of Hartshorne-Hirschowitz and Sernesi \cite{Se} and we recall a few facts. We fix a nodal curve $X\subset \PP^r$ with $p_a(X)=g$ and $\mbox{deg}(X)=d$, then denote by $T_X^1$ the \emph{Lichtenbaum-Schlessinger} sheaf defined via the exact sequence
$$0\longrightarrow T_X\longrightarrow T_{\PP^r|X}\longrightarrow N_{X/\PP^r} \longrightarrow T_X^1 \longrightarrow 0.$$
Setting $N_X':=\mbox{Ker}\{N_{X/\PP^r}\rightarrow T_X^1\}$, the vanishing $H^1(X, N_X')=0$ is a sufficient condition for $X\subset \PP^r$ to be flatly smoothable and for $\mathrm{Hilb}_{g, r, d}$ to be smooth and of expected dimension $(r+1)d-(r-3)(g-1)$ at the point $[X]$, cf. \cite{Se} Proposition 1.6.

\vskip 3pt

Suppose $X:=C\cup_{\Delta} D$ is the union of two smooth curves $C, D\subset \PP^r$, meeting transversally at a set of points $\Delta:=\{p_1, \ldots, p_{\delta}\}$. From \cite{Se} Lemma 5.1, one writes the following exact sequence on $X$
\begin{equation}\label{ser}
0\longrightarrow N_{D/\PP^r}\Bigl(-\sum_{i=1}^{\delta} p_i\Bigr)\longrightarrow N_X'\longrightarrow N_{C/\PP^r}\longrightarrow 0.
\end{equation}
If both $H^1(C, N_{C/\PP^r})=0$ and $H^1(D, N_{D/\PP^r}(-p_1-\cdots-p_{\delta}))=0$, then $H^1(X, N_X')=0$ and $X$ is flatly smoothable in $\PP^r$. The next result is essentially contained in \cite{Se}:

\begin{lem}\label{attach}
Suppose $C\subset \PP^r$ is a non-special smooth curve of genus $g$ and $p_1, \ldots, p_{\delta}\in C$ are distinct points in general linear position. Let $R\subset \PP^r$ be a smooth rational curve of degree $s$, where $\delta-1\leq s\leq r$, intersecting $C$ transversally at the points $p_1, \ldots, p_{\delta}$. Then the union $X:=C\cup R$ is a flatly smoothable non-special nodal curve in $\PP^r$ satisfying $H^1(X, N_X')=0$.
\end{lem}
\begin{proof}
Under the isomorphism $\nu_s:\PP^1\stackrel{\cong}\longrightarrow R\subset \PP^s\subset \PP^r$ (hence $\nu^*_s(\OO_R(1))=\OO_{\PP^1}(s)$), it is well-known that $N_{R/\PP^r}=\OO_{\PP^1}(s+2)^{\oplus (s-1)}\oplus \OO_{\PP^1}(s)^{\oplus (r-s)}$. Then
$$H^1\bigl(R, N_{R/\PP^r}(-p_1-\cdots-p_{\delta})\bigr)=H^1\bigl(\OO_{\PP^1}(s+2-\delta)\bigr)^{\oplus (s-1)}\oplus H^1\bigl(\OO_{\PP^1}(s-\delta)\bigr)^{\oplus (r-s)}=0.$$
Since $C$ is non-special, $H^1(C, N_{C/\PP^r})=0$ and from  (\ref{ser}) it follows that $X$ is smoothable in $\PP^r$. From the exact sequence
$$\cdots\longrightarrow H^1\Bigl(R, \OO_R(1)(-\sum_{i=1}^{\delta} p_i)\Bigr)\longrightarrow H^1(X, \OO_X(1))\longrightarrow H^1(C, \OO_C(1))\longrightarrow \cdots,$$
we obtain that $X$ is non-special precisely when $\delta\leq s+1$.
\end{proof}

\vskip 4pt

We turn our attention to the Ideal Generation Conjecture for $(L, V)\in G^r_d(C)$.  Via Proposition \ref{mrcc}, this is equivalent to the generic vanishing statements
\begin{equation}\label{igc1}
H^0(C, M_V\otimes \xi)=0, \ \ \mbox{ for a general } \ \  \xi \in \mathrm{Pic}^{g-1+\lfloor \frac{d}{r}\rfloor}(C), \ \ \mbox{ and}
\end{equation}
\begin{equation}\label{igc2}
H^0\Bigl(C, \bigwedge^{r-1} M_V\otimes \xi\Bigr)=0, \mbox{ for a general } \ \  \xi \in \mathrm{Pic}^{g-1+d-\lceil \frac{d}{r}\rceil}(C).
\end{equation}
We shall prove this for a nodal curve in $\PP^r$ obtained by attaching at most $r-1$ general secant lines to a smooth curve of integral slope.
\vskip 4pt

\noindent \emph{Proof of Theorem \ref{igc}.} We fix positive integers $g, r$ and $d$
such that $\rho:=\rho(g, r, d)\geq 0$ and set $d_1:=d-r\lfloor \frac{d}{r} \rfloor<r$ and $g_1:=\mbox{max}\{g-d_1, 0\}$.
By direct computation, we find $\rho(g_1, r, \lfloor \frac{d}{r}\rfloor r)\geq \mbox{min}\{\rho-d_1, 0\}$. This last quantity is non-negative whenever $\rho\geq r$. In this case, by using Theorem \ref{MRC0}, we can construct a smooth curve $C_1\subset \PP^r$ of genus $g_1$ and degree $r\lfloor \frac{d}{r}\rfloor$ with general moduli and with the bundle $M_{C_1}$ verifying condition (R).
When on the other hand $0\leq \rho \leq r-1$, then $s:=g-d+r\geq 0$ and one writes
$$g=rs+s+\rho \ \ \mbox{ and  } \ \ d=rs+r+\rho.$$
Observe that $\rho(rs+s, r, rs+r)=0$ and use again Theorem \ref{MRC0} to choose a curve $C_1\subset \PP^r$ of genus $rs+s$ and degree $rs+r$ enjoying the exact same properties as above.
\vskip 3pt

To summarize the two cases, one can find integers $a\geq 1$ and $0\leq d_1\leq r-1$ such that
 $$g=g_1+d_1 \ \ \mbox{ and } \ \ d=ar+d_1,$$ for which there exists  a smooth curve with general moduli $C_1\subset \PP^r$ with $\mbox{deg}(C_1)=ar$ and $g(C_1)=g_1$,
 such that $M_{C_1}$ verifies condition (R). To $C_1$ we attach $d_1$ general two-secant lines $\ell_1, \ldots, \ell_{d_1}\subset \PP^r$. The resulting nodal curve
 $$X:=C_1\cup \ell_1\cup \cdots \cup \ell_{d_1}$$ has
 $\mbox{deg}(X)=d$ and $p_a(X)=g$, and is flatly smoothable in $\PP^r$ to a curve with general moduli due to a repeated application of Lemma \ref{attach}. It remains to check conditions (\ref{igc1}) and (\ref{igc2}) and we explain only the first part, omitting the details for the second. We pick a line bundle $\eta \in \mbox{Pic}^{g_1-1+a}(C_1)$ such that $H^0(C_1, M_{C_1}\otimes \eta)=0$; the existence of such $\eta$ is implied by the property (R). We create a line bundle $\xi$ on the curve $X$ such that $\xi_{\ell_j}$ is of degree $-1$ for each $j=1, \ldots, d_1$, whereas
 \begin{equation}\label{assC}
 \xi_{C_1}=\eta\otimes \OO_{C_1}\bigl(\sum_{j=1}^{d_1} \ell_j\cdot C_1\bigr).
\end{equation}
 We claim that $H^0(X, M_X\otimes \xi)=0$. This indeed follows by tensoring and taking cohomology in  the Mayer-Vietoris sequence on $X$, while using (\ref{assC}), together with the fact that since $M_{\ell_j}=\OO_{\PP^1}(-1)\oplus \OO_{\PP^1}^{\oplus (r-1)}$, one has that $H^0(\ell_j, M_{\ell_j}\otimes \xi_{\ell_j})=0$.

 Finally, note that $g-1+\lfloor \frac{d}{r}\rfloor=g-1+a=\mbox{deg}(\xi)$, which shows that $\xi$ has precisely the correct degree to establish IGC.
 \hfill $\Box$

A variation of this idea gives a proof of MRC for general curves of degrees that are congruent to $\pm 1$ modulo $r$.

\vskip 3pt

\begin{thm}\label{pm} Let $C$ be a general curve of genus $g$ and fix positive integers $r, \mu$ and $d:=\mu r\pm 1$. Then the Minimal Resolution Conjecture holds for a general embedding
$C\hookrightarrow \PP^r$ of degree $d$.
\end{thm}
\begin{proof} We treat only the case $d=\mu r+1$, the other case being similar. From Brill-Noether theory, we obtain that $g\leq (r+1)(\mu-1)+1$. Applying Theorem \ref{MRC0}, there exists a smooth curve with general moduli $C_1\subset \PP^r$ of genus $g-1$ and degree $d-1=\mu r$, such that the kernel bundle $M_{C_1}$ enjoys property (R).

 Let $\ell$ be a general $2$-secant line to $C_1$ and set $X:=C_1\cup \ell\subset \PP^r$. It is easy to verify that $H^0(X, \OO_X(1))\cong H^0(C_1, \OO_{C_1}(1))$ and $H^0(X, \omega_X(-1))\cong H^0(C_1, K_{C_1}(-1))$, so the Petri map $\mu_0(X)$ can be assumed to be injective and $X$ deforms in $\PP^r$ to a curve of genus $g$ with general moduli. By assumption, $C_1$ possesses for each $1\leq i\leq r-1$ a line bundle $\eta\in \mbox{Pic}^{g-2+i\mu}(C_1)$ such that $H^0\bigl(C_1,\bigwedge^i M_{C_1}\otimes \eta \bigr)=0$. Observing that for all $i\leq r-1$
 $$g-1+\Bigl \lfloor \frac{id}{r}\Bigr \rfloor =g-1+i\mu$$ (and this is the point where the assumption $d\equiv 1 \mbox{ mod } r$ is essential!), we can construct a line bundle $\xi \in \mbox{Pic}^{g-1+i\mu}(X)$, such that $\xi_{\ell}=\OO_{\PP^1}(-1)$ and $\xi_{C_1}=\eta(C_1\cdot \ell)$. Now one checks directly that $H^0\bigl(X, \bigwedge^i M_X\otimes \xi)=0$, thus finishing the proof.
\end{proof}

\vskip 2pt

After this preparations, we are finally ready to prove Theorem \ref{MRC1}.

\vskip 4pt
\noindent \emph{Proof of Theorem \ref{MRC1}.} We fix $d, r\geq 1$ such that $d\geq 2r$. Using Theorems \ref{MRC0} and \ref{pm}, we need to consider only the case when $\frac{d}{r}\not\equiv 0, \pm 1 \mbox{ mod } r$ and inequality (\ref{felt}) holds. We set $a:=\lfloor \frac{d}{r}\rfloor -2\geq 0$ and write $d=ar+d_1$, where $2r+2\leq d_1\leq 3r-2$. We set $g_1:=\mbox{max}\{g-ar, 0 \}$. Inequality (\ref{felt}) implies that $d_1\geq 2g_1-2$. If $d_1$ is even, applying Proposition \ref{bielliptic}, there exists a smooth non-special curve $C_1\subset \PP^r$ of genus $g_1$ and degree $d_1$, such that $\Omega_{\PP^r|C_1}^1$ verifies condition (R). If, on the other hand, $d_1$ is odd, then there is a curve of degree $d_1-1$ and genus $g_1$ with the same property. We treat only the case when $d_1$ is even and indicate at the end the modifications in the proof needed in the remaining case.

Setting, as usual, $M_{C_1}:=\Omega_{\PP^r |C_1}^1(1)$, condition (R) amounts to the following  vanishing
\begin{equation}\label{feltetel}
H^0\Bigl(C_1, \bigwedge^i M_{C_1}\otimes \eta\Bigr)=0, \ \mbox{ for }i=1, \ldots, r-1 \ \mbox{ and a general  } \eta\in \mbox{Pic}^{g_1-1+\lfloor \frac{id_1}{r}\rfloor}(C_1).
\end{equation}

\vskip 3pt

To $C_1$ we attach $a$ rational normal curves as follows. We fix subsets $\Delta_1, \ldots, \Delta_a \subset C_1$ consisting of general points such that
$|\Delta_j|\leq r+1$ for $j=1, \ldots, a$ and furthermore $g=g_1+\sum_{j=1}^a |\Delta_j|-a$. For each $1\leq j\leq a$, we choose a general rational curve
$R_j\subset \PP^r$ intersecting $C_1$ transversally along the set $\Delta_j$, then set
$$X:=C_1\cup R_1\cup \ldots \cup R_a\subset \PP^r.$$
Clearly $p_a(X)=g_1+\sum_{j=1}^a |\Delta_j|-a=g$ and $\mbox{deg}(X)=d$. Applying Lemma \ref{attach}, we conclude that $X$ is non-special and flatly smoothable in $\PP^r$.

Let us fix an index $1\leq i\leq r-1$. Via the surjection
$$\mbox{Pic}^{g-1+\lfloor \frac{id}{r}\rfloor}(X)\longrightarrow \mbox{Pic}^{g-1+a+\lfloor \frac{id_1}{r} \rfloor}(C_1) \times \prod_{j=1}^a \mbox{Pic}^{i-1}(R_j) \longrightarrow 0,$$
we consider a line bundle $\xi$ on $X$ of degree $g-1+\lfloor \frac{id}{r}\rfloor$, such that $\mbox{deg}(\xi_{R_j})=i-1$, for all $j$. We claim that $\xi_{C_1}$  can be chosen so that $H^0\bigl(X, \bigwedge^i M_X\otimes \xi\bigr)=0$.
 \vskip 3pt

Indeed, we first observe that $\bigwedge^i M_{R_j}$ is a sum of line bundles of degree $-i$, hence $H^0\bigl(R_j, \bigwedge^i M_{R_j}\otimes \xi_{R_j}\bigr)=0$ for degree reasons. Considering the inclusion
$$H^0\Bigl(X, \bigwedge^i M_X\otimes \xi\Bigr)\hookrightarrow H^0\Bigl(C_1, \bigwedge^i M_{C_1}\otimes\xi_{C_1}\Bigr) \oplus \Bigl( \bigoplus_{j=1}^a H^0\Bigl(R_j, \bigwedge^i M_{R_j}\otimes \xi_{R_j}\Bigr)\Bigr)$$
induced by the Mayer-Vietoris sequence on $X$,
from the previous observation it follows that a non-zero section in $H^0\bigl(X, \bigwedge^i M_X\otimes \xi\bigr)$ corresponds to a non-zero section in $H^0\bigl(C_1, \bigwedge^i M_{C_1}\otimes \xi_{C_1}(-\sum_{j=1}^a \Delta_j)\bigr)$. Observing that $\mbox{deg}(\xi_{C_1})-\sum_{j=1}^a |\Delta_j|=g_1-1+\lfloor \frac{id_1}{r}\rfloor$, we choose $\xi_{C_1}$ so that the vanishing (\ref{feltetel}) holds for $\eta=\xi_{C_1}\bigl(-\sum_{j=1}^a \Delta_j\bigr)$. We conclude that the kernel bundle of a general smoothing of $X\subset \PP^r$ verifies condition (R).
\hfill $\Box$

\begin{rmk} In the previous proof, if $d_1$ is odd, then we start with a smooth curve of degree $d_1-1$ and genus $g_1$, to which we attach as before $a-1$ rational normal curves and one \emph{linearly normal} elliptic curve $E\subset \PP^r$. Since the restricted cotangent bundle  $\Omega_{\PP^r|E}^1$ is stable, the rest of the proof follows along similar lines.
\end{rmk}

\section{Ulrich bundles on $K3$ surfaces}

Let $X\subset \PP^r$ be a smooth arithmetically Cohen-Macaulay projective variety of degree $d$. A vector bundle $E$ on $X$ is said to be an \emph{Ulrich sheaf} if
\begin{equation}\label{ulr3}
H^i(X, E(-i))=0 \ \mbox{ for } i>0 \ \mbox{ and } \ H^i(X, E(-i-1))=0\  \mbox{ for } \  i<\mathrm{dim}(X).
\end{equation}
This definition is equivalent to the one mentioned in the Introduction, see \cite{ES} Proposition 2.1. An Ulrich sheaf $E$  enjoys a number of properties we list, see \cite{CHGS}:

\vskip 3pt

\noindent (i) The restriction $E_H$ to a general hyperplane section $H$ of $X$ is again an Ulrich bundle.
\vskip 2pt

\noindent (ii) $h^0(X, E)=d\cdot \mbox{rk}(E)$ and $\mbox{deg}(E|_C)=\mbox{rk}(E)(d+g-1)$, where $g$ is the genus of a general curvilinear section $C=X\cap \PP^{r-\mathrm{dim}(X)+1}$ of $X$.
Furthermore, $\OO_C(-1)\notin \Theta_{E_C}$, hence the restriction $E_C$ admits a theta divisor.

\vskip 2pt

\noindent (iii) Ulrich bundles are semistable with respect to the polarization $\OO_X(1)$.

\vskip 3pt

Combining properties (i) and (iii), one obtains rational maps between moduli spaces of semistable bundles on $X$ and on the hyperplane section $H$ respectively.
From now on let $X=S$ be a smooth surface, in which case the condition (\ref{ulr3}) amounts to the vanishing of the following cohomology groups
$$
H^0(S,E(-1)),\ H^1(S,E(-1)),\ H^1(S,E(-2)),\ H^2(S,E(-2)).
$$
This implies the further vanishing $H^0(S,E(-2))=0$ and $H^2(S,E(-1))=0$ (the bundle $E$ being $0$-regular, is $1$-regular as well), hence
$\chi(S,E(-1))=\chi(S,E(-2))=0$, see also \cite{ES} Corollary 2.2.
Applying Riemann-Roch to both $E(-1)$ and $E(-2)$ and taking the difference of the Euler characteristics, we obtain the relation
\begin{equation}\label{ulrk}
H\cdot \left(c_1(E)-\frac{\mathrm{rk}(E)}{2}(K_S+3H)\right)=0.
\end{equation}
This calculation motivates the following:

\begin{defn}
A \emph{special Ulrich bundle} on a surface $S$ is a $0$-regular rank $2$ vector bundle $E$  with determinant $\mbox{det}(E)=K_S(3)$.
\end{defn}

It is proved in \cite{ES} Corollary 2.3 that such bundles are indeed Ulrich. If  $E$ is a special rank $2$ Ulrich bundle on a $K3$ surface $S\subset \PP^{s+1}$, from Riemann-Roch we compute $c_2(E)=\frac{5}{2}H^2+4$. Moreover,  $E$ being $0$-regular it is globally generated. A parameter count performed in \cite{ES} Remark 6.4 suggests that $K3$ surfaces could possess rank $2$ Ulrich bundles. Our Theorem \ref{thm: ulrich} confirms this expectation and we show that the hypothesis of \cite{ES} Proposition 6.2 is verified for Lazarsfeld-Mukai vector bundles on many polarized $K3$ surfaces.
An immediate consequence of the relation (\ref{ulrk}) is the following fact:

\begin{cor}\label{oddrk}
A $K3$ surface with Picard number $1$ carries no Ulrich bundles of odd rank.
\end{cor}

In even rank, for each $a\geq 1$ one looks for  Ulrich bundles $E$ on $S$ with
\begin{equation}\label{chernu}
\mbox{rk}(E)=2a, \  \ \mbox{det}(E)=\mathcal O_S(3a)\mbox{ and } c_2(E)=9a^2s-4a(s-1).
\end{equation}
If $\mbox{Pic}(S)=\mathbb Z\cdot [H]$ with $H^2=2s$, every Ulrich bundle $E$ on $S$  satisfies (\ref{chernu}). Natural candidates for $E$ are the Lazarsfeld-Mukai bundles $E_{C, A}$ defined by the exact sequence
\begin{equation}
\label{eqn: F}
0\longrightarrow E_{C, A}^{\vee} \longrightarrow H^0(C, A)\otimes \OO_S\longrightarrow A\longrightarrow 0,
\end{equation} where $C\in |\OO_S(3a)|$ is a smooth curve and $A\in W^{2a-1}_{9a^2s-4a(s-1)}(C)$ is complete and base point free. The curve $C$ has $K_C=\OO_C(3a)$ and $\mbox{Cliff}(C)\leq \mbox{Cliff}(\OO_C(1))=6as-2s-2$, with equality for instance when $\mbox{Pic}(S)=\mathbb Z\cdot [H]$. Note that it is by no means certain that such an $A$ exists, and if so, that it leads to an Ulrich bundle. Theorem \ref{thm: ulrich} establishes these facts in the most important case, $a=1$. Then we construct stable Ulrich bundles of higher even rank, by deforming
bundles sitting in extensions of two Ulrich bundles of smaller rank.

\begin{rmk}
The hypothesis in Theorem \ref{thm: ulrich} that the Clifford index of a cubic section $C\in |\OO_S(3)|$ be computed by $\OO_S(1)$ is rather mild.
For instance, it is satisfied if $S\subset \PP^3$ is a quartic surface, when $C$ is a $(3, 4)$ complete intersection in $\PP^3$. From \cite{L2} page 185 we obtain that $\mathrm{gon}(C)\ge 8$. Since $C$ has Clifford dimension $1$, cf. \cite{CP95}, it follows that $\mbox{Cliff}(C)\geq 6=\mathrm{Cliff}({\mathcal O}_C(1))$. The only place in the proof where this condition is used is to ensure that $C$ carries a base point free pencil of degree $\frac{5H^2}{2}+4$.
\end{rmk}

\begin{lem}
\label{lemma: AF11}
Let (S, H) be a polarized $K3$ surface of genus $g$ and $C\in |H|$ a general curve in its linear system having gonality $k$. Then $C$ carries a complete, base point free pencil $\mathfrak g^1_{g-k+3}$.
\end{lem}

\proof
The case $\rho(g, 1, k)>0$ follows immediately, for in this situation $g=2k-3$ and hence $g-k+3=k$. We may assume that $\rho(g, 1, k)\leq 0$. When the Clifford dimension of a general curve in $|H|$ equals $1$, from \cite{AF11} Theorem 3.12 and Remark 3.13, one obtains that for a general  $C\in |H|$, every component of $W^1_{g-k+2}(C)$ has dimension $g-2k+2$. Via excess linear series it then follows that each component of $W^1_{g-k+3}(C)$ has dimension $\mbox{dim}(W_{g-k+2}^1(C))+2=g-2k+4\bigl(=\rho(g, 1, g-k+3)\bigr)$.

Since $\mathrm{dim}\bigl(C+W^1_{g-k+2}(C)\bigr)=g-2k+3$, we conclude that the general element in every component of $W^1_{g-k+3}(C)$ is base point free and complete. The case when the general curve in $|H|$ has Clifford dimension at least $2$, will not be needed in this paper,  but it can be deduced along the lines of \cite{AF11} Section 5.
\endproof

\vskip 2pt

The following result is needed in the proof of Theorem \ref{thm: ulrich}.
\vskip 3pt

\begin{lem}\label{quadr}
Let $(S, H)$ be a polarized $K3$ surface with $H^2=2s\geq 4$ and $D\in |2H|$ a smooth quadric section. Then the following estimate holds
$$\mathrm{dim} \ \bigl\{\Gamma\in D_{5s+4}: h^0(D, \OO_D(\Gamma-H))\geq 1\bigr\}\leq 2s+7.$$
\end{lem}
\begin{proof}
By direct calculation, $\varphi_H:D\hookrightarrow \PP^{s+1}$ is a smooth half-canonical curve with $\mbox{deg}(D)=D\cdot H=4s$ and $g(D)=4s+1$. We consider the incidence variety
$$\V:=\Bigl\{(\Gamma, \zeta)\in D_{5s+4}\times D_{s+4}: \Gamma \in |\OO_D(H+\zeta)| \Bigr\},$$
together with the projections $\pi_1:\V\rightarrow D_{5s+4}$ and $\pi_2:\V\rightarrow D_{s+4}$. Note that $\pi_1(\V)$ is precisely the variety whose dimension we have to compute. To estimate $\mbox{dim}(\V)$ we look at the fibres of $\pi_2$. By Riemann-Roch, $h^0(D, \OO_D(H+\zeta))=h^0(D, \OO_D(H-\zeta))+s+4$, for every $\zeta\in D_{s+4}$. In particular, for a general divisor $\zeta\in D_{s+4}$,  we obtain that $\pi_2^{-1}(\zeta)=\PP H^0(\OO_D(H+\zeta))\cong \PP^{s+3}$, and hence $\V$ has a unique irreducible component of dimension $2s+7$ that dominates $D_{s+4}$.

For $i\geq 1$, the locally closed variety  $\Sigma_i:=\{\zeta\in D_{s+4}: h^0(D, \OO_D(H-\zeta))=i\}$ has dimension at most $\mbox{dim } |\OO_D(H)|-i+1=s+2-i$. If $\zeta \in \Sigma_i$ then
$\pi_2^{-1}(\zeta)\cong \PP^{s+i+3}$, hence $\mbox{dim } \pi_2^{-1}(\Sigma_i)\leq \mbox{dim } \Sigma_i+s+i+3\leq 2s+5.$
To sum up,  all components of $\V$ are of dimension $\leq 2s+7$, implying the same conclusion for $\mbox{dim}(\pi_1(\V)).$
\end{proof}

\vskip 3pt
We now proceed to show that polarized $K3$ surfaces satisfying a mild Brill-Noether genericity condition carry stable rank $2$ Ulrich bundles.
\vskip 3pt

\noindent \emph{Proof of Theorem \ref{thm: ulrich}.}
We start with a $K3$ surface $S\subset \PP^{s+1}$ and let $H\in |\OO_S(1)|$ be a hyperplane section with $H^2=2s$. We fix a smooth curve  $C\in |\mathcal O_S(3)|$ and compute its genus $g(C)=9s+1$.
Invoking \cite{CP95}, note that $C$ has Clifford dimension $1$ and clearly $\mathrm{Cliff}(\mathcal O_C(1))=4s-2$. Our hypothesis implies $\mathrm{gon}(C) = 4s$, hence by Lemma
\ref{lemma: AF11}, $C$ possesses a complete base point free pencil $A\in W^1_{5s+4}(C)$. The candidate Ulrich bundle is the Lazarsfeld-Mukai bundle $E:=E_{C,A}$. More precisely, the general point $(C,A)$ of  any dominating component $\mathcal W$ of the relative space $\mathcal W^1_{5s+4}(|\mathcal O_S(3)|)$ over the linear system $|\mathcal O_S(3)|$ corresponds to a  complete and base point free pencil $\mathfrak g^1_{5s+4}$.

Since the Ulrich condition (\ref{ulr3}) is open, we need to ensure that the \emph{non-Ulrich locus} does not coincide with the whole $\mathcal W$.

\vskip 3pt

{\em Step 1.} For a general point $(C,A)\in \mathcal W$, we verify the partial Ulrich condition
\begin{equation}
\label{eqn: Ulrich}
H^0(S, E_{C,A}(-1))=0.
\end{equation}

We shall find an explicit parametrization of the failure locus of (\ref{eqn: Ulrich}) and count parameters. Consider the following Grassmann bundle over the moduli space of LM bundles
$$
\mathcal G:=\Bigl\{(E_{C,A},\Lambda):  (C,A)\in \mathcal W,\ \Lambda\in G\bigl(2,  H^0(S, E_{C,A})\bigr)\Bigr\}.
$$

Recall from \cite{L} or \cite{AF11} the following dimension estimate
$$\mathrm{dim}(\mathcal W) \ge \mathrm{dim}|\mathcal O_S(3)|+\rho(9s+1,1,5s+4)=10s+6.$$

Since the projection $\mathcal G\rightarrow \mathcal W$ is dominant with fibre $\PP H^0(S, E_{C, A}\otimes E_{C, A}^\vee )$ over a general point $(C, A)\in \W$, the estimate
$\mathrm{dim}(\mathcal G) \ge 10s+6$ holds as well. Since $h^0(S,E_{C, A})=h^0(C, A)+h^1(C, A)=4s$, the dimension of the space  of LM bundles $E_{C, A}$ corresponding to pairs $(C,A)\in\mathcal W$ has dimension at least $\mbox{dim}(\G)-\mathrm{dim}\ G(2,H^0(E_{C,A}))\geq 2s+10$. Observe that the local dimension at $E_{C, A}$ of the moduli space $\mathrm{Spl}(2, \OO_S(3), 5s+4)$ of simple vector bundles of rank $2$ on $S$ with first Chern class $\mathcal O_S(3)$ and second Chern class $5s+4$ is also equal to $c_1^2(E)-2\mbox{rk}(E) \chi(E)+2\mbox{rk}(E)^2+2=2s+10$, that is, a general point of $\mathrm{Spl}(2, \OO_S(3), 5s+4)$ corresponds to a bundle $E_{C, A}$.
\vskip 3pt

 Next, we consider the projective bundle
$$
\mathcal P:=\bigl\{(E_{C,A}, \ell):(C,A)\in\mathcal W, \ \ell\in \PP H^0(S,E_{C,A}) \bigr\},
$$
with $\mbox{dim}(\P)\geq 6s+9$.
Any LM bundle $E_{C,A}$ is given by an extension
\[
0\longrightarrow \mathcal O_S \stackrel{\ell} \longrightarrow E_{C,A}\longrightarrow \mathcal I_{\Gamma/S}(3)\longrightarrow 0,
\]
where $\Gamma\in S^{[5s+4]}$ is a 0-dimensional subscheme which satisfies the \emph{Cayley-Bacharach} (CB) condition with respect to $|\mathcal O_S(3)|$. This condition is necessary in order to obtain locally free extensions, cf. \cite{L2} page 177. Note that
\[
\mathrm{dim}\ \mathrm{Ext}^1(\mathcal I_{\Gamma/S}(3),\mathcal O_S)=1;
\]
indeed, from the exact sequence defining $\Gamma$ and from $h^0(S,E_{C,A}^{\vee})=h^1(S,E_{C,A})=0$, we obtain
an isomorphism $H^0(S,\mathcal O_S)= \mathrm{Ext}^1(\mathcal I_{\Gamma/S}(3),\mathcal O_S)$. In particular, $\Gamma$ determines uniquely the LM bundle $E_{C,A}$ and the map $\varphi:\P\rightarrow S^{[5s+4]}$ given by $\varphi([E_{C, A}, \ell]):=\Gamma$ is generically injective onto its image.

\vskip 3pt

Since $H^0(S,E_{C,A}(-1))\cong H^0(S,\mathcal I_{\Gamma/S}(2))$, we shall show that cycles  $\Gamma\in \mbox{Im}(\varphi)$  with $H^0(S,\mathcal I_{\Gamma/S}(2))\ne 0$
depend on at most $6s+8 \leq \mathrm{dim}(\mathcal P)-1$ parameters. To this end, we consider the incidence variety, see also \cite{CKM12} Proposition 3.18 for the case $s=2$,
$$
\mathcal Z:=\Bigl\{(D,\Gamma):D\in|\mathcal O_S(2)|,\ \Gamma\subset D \mbox{ satisfies CB with respect to } |\mathcal O_S(3)|\Bigr\}.
$$

We fix a smooth section $D\in |\mathcal O_S(2)|$ and an effective divisor $\Gamma\in D_{5s+4}$. For a point $p\in \mathrm{supp}(\Gamma)$, we write $\Gamma=\Gamma_p+p$, where $\Gamma_p\in D_{5s+3}$. Via the exact sequence
 $$0\longrightarrow H^0(S, \OO_S(1))\stackrel{+D}\longrightarrow H^0(S, \I_{\Gamma/S}(3))\longrightarrow H^0(D, \OO_D(3H-\Gamma))\longrightarrow 0,$$
 we rephrase the Cayley-Bacharach condition for $\Gamma$ as requiring that the isomorphism  $H^0(D,\mathcal O_D(3H-\Gamma))\cong H^0(D, \OO_D(3H-\Gamma_p))$ hold, or equivalently by Riemann-Roch,
 $$h^0(D, \OO_D(\Gamma_p-H))=h^0(D, \OO_D(\Gamma-H))-1, \ \ \ \mbox{ for each } p\in \mathrm{supp}(\Gamma).$$
In particular, $h^0(D, \OO_D(\Gamma-H))\geq 1$. Via Lemma \ref{quadr}, we conclude that the dimension of each fibre of the map $\cZ\rightarrow |\OO_S(2)|$ does not exceed $2s+7$; thus $\mbox{dim}(\cZ)\leq \mbox{dim } |\OO_S(2)|+2s+7=6s+8$, which establishes condition (\ref{eqn: Ulrich}) for a general $(C, A)\in \W$.

\medskip

{\em Step 2.} The partial Ulrich condition (\ref{eqn: Ulrich}) implies the full Ulrich condition (\ref{ulr3}).

\vskip 4pt

By Serre duality, using the isomorphism $E^\vee\cong E(-3)$, the condition (\ref{ulr3}) reduces to $H^0(S,E(-1))=H^1(S,E(-1))=0$. Twisting the sequence (\ref{eqn: F}) by $\mathcal O_S(2)$ and taking cohomology, we obtain the exact sequence
\[
0\to H^0(S,E(-1))\rightarrow H^0(C,A)\otimes H^0(S,\mathcal O_S(2))\rightarrow  H^0(C,A(2)) \rightarrow  H^1(S, E(-1))\to 0,
\]
and an isomorphism
$$
H^1(C,A(2))=H^0(S,E(-2))^{\vee}=0,
$$ from (\ref{eqn: Ulrich}); hence, for a general pair $(C,A)\in \W$, the bundle $A(2)$ is non-special, which implies $h^0(C,A(2))=8s+4$. Since $h^0(S,\mathcal O_S(2))=4s+2$, we obtain $H^1(S, E(-1))=0$.

\medskip

We have proved that the failure locus of the Ulrich condition for $E_{C,A}$ is a genuine effective divisor in $\mathcal W$ (or rather in the open subset of $\mathcal W$ given by the vanishing of $H^0(S,E_{C,A}(-2))$).

\medskip

{\em Step 3.} We prove that $E=E_{C, A}$ is stable. Simplicity already follows from \cite{AF11} Remark 3.13 and suppose $E$ is not stable. Following \cite{CHGS} Theorem 2.9, any such $E$ is presented as an extension
\[
0\longrightarrow M\longrightarrow E\longrightarrow N\longrightarrow 0,
\]
where $M,N$ are Ulrich line bundles. Since $\chi(S,M(-1))$, $\chi(S,M(-2))$, $\chi(S,N(-1))$ and $\chi(S,N(-2))$ vanish, we obtain the following numerical conditions: $$M^2=N^2=4s-4 \ \ \mbox{ and } M\cdot H=N\cdot H=3s.$$ Furthermore, $M\otimes N=\mathcal O_S(3)$ and $M\not\cong N$, because the bundle $E$ is simple. In particular, $h^0(S,M\otimes N^\vee)=h^0(S,N\otimes M^\vee)=0$ which implies that $h^1(S,M\otimes N^\vee)=2s+6$. Hence $\dim\ \PP(\mathrm{Ext}^1(N,M))=2s+5$.
Since the space of special Ulrich bundles has dimension $2s+10$ and $\mbox{Pic}(S)$ is discrete, we conclude that a general $E$ is stable.
\endproof

\subsection{Existence of Ulrich stable bundles of even rank.}

We now restrict ourselves to a K3 surface $S\subset \PP^{s+1}$ with Picard number 1. By Corollary \ref{oddrk}, $S$ does not admit Ulrich bundles of odd rank.
 Recall that if $E$ is an Ulrich bundle on $S$ of even rank $2a$, then $\mbox{det}(E)= \OO_S(3a) $ and $c_2(E)= 9a^2s - 4a(s-1)$. We record the following consequence of Riemann-Roch, see also
\cite{CHGS} Proposition 2.12.

\begin{lem}\label{ext}
Let $E$ and $F$ be Ulrich bundles on $S$ of ranks $2a$ and $2b$ respectively. Then
$$\chi(S, E^{\vee} \otimes F) = -2abs - 8ab.$$
\end{lem}
We now show that such an $S$ carries stable Ulrich bundles of every even rank.

\begin{thm}\label{highrank}
For any $K3$ surface $S\subset \PP^{s+1}$ with $\mathrm{Pic}(S)=\mathbb Z\cdot [H]$ and any integer $a\geq 1$, there exists an $(8a^2+2a^2s+2)$-dimensional family of stable Ulrich bundles on $S$ of rank $2a$.
\end{thm}
\begin{proof}
Our proof follow closely the lines of \cite{CHGS} Theorem 5.7 and especially those of \cite{C} Theorem 3.1. We proceed by induction on $a$. The case $a=1$ is part of Theorem \ref{thm: ulrich}. Suppose there exist stable Ulrich bundles of any even rank smaller than $2a$. We choose stable Ulrich bundles $F_1$
and $F_2$ of ranks $2$ and $2a-2$ respectively, with $F_1\ncong F_2$ in the case $a=2$. From Lemma \ref{ext}, we find that $\dim \ \Ext^1(F_1,F_2) >0$, that is, a general extension $E\in \PP \mbox{Ext}^1(F_2, F_1)$ is a \emph{simple} Ulrich bundle of rank $2a$, cf. \cite{CHGS} Lemma 4.2.
Let $\mbox{Def}(E)$ be the universal deformation space of $E$; since the deformation functor for simple bundles is pro-representable, $\mbox{Def}(E)$ can be constructed in the \'etale topology and its dimension equals
\begin{equation} \label{bound1}
\mathrm{dim } \ \mathrm{Def}(E)= h^1(S, E^{\vee} \otimes E)  =  2- \chi(S, E^{\vee} \otimes E) = 2 + 2a^2s +8a^2 .
\end{equation}

By a parameter count, we show that the general element of $\mbox{Def}(E)$ corresponds to a stable bundle. Since the Jordan-H\"older filtration of a strictly semistable Ulrich bundle is a direct sum of stable Ulrich bundles (of the same slope), we are led to count the number of moduli of vector bundles appearing in this way. We fix stable, pairwise non-isomorphic Ulrich bundles $E_1, \ldots , E_n$ where $\rk(E_i) =2a_i$ for $i=1, \ldots, n$. By (\ref{chernu}), the Chern classes of each $E_i$ are uniquely determined by their ranks. We consider the family $\cU$ of Ulrich bundles $E$ of rank $2a$ such that
$$\mbox{gr}(E)=E_1^{\oplus k_1}\oplus \cdots
\oplus E_n^{\oplus k_n}.$$ Thus we have $a= k_1 a_1 +
\cdots + k_n a_n$ and let $m:=k_1 + \cdots + k_n$ be the total number of bundles employed.

Let $F$ be an Ulrich bundle of rank $2b$ constructed from successive extensions of some of the $E_i$'s, such that for a given $j\in \{1, \ldots, n\}$, precisely $k$  copies of the bundle $E_j$ are used.  Since $\mbox{Hom}(E_i, E_j)=0$ for $i\neq j$, we find that $h^0(S, E_j^{\vee} \otimes F) \leq k$
and similarly, by Serre duality, $h^2(S, E_j^{\vee} \otimes F) \leq k $. Thus,
\begin{equation} \label{bound2}
 \dim \ \Ext^1(E_j, F) \leq k - \chi(S, E_j^{\vee} \otimes F) = k + (2s+8)a_j b.
 \end{equation}

Using the bounds \eqref{bound1} and \eqref{bound2}, we conclude that the space $\cU$ of Ulrich bundles obtained by a succession of extensions involving $k_i$ times the bundle $E_i$ for $i=1, \ldots, n$ has dimension at most
 $$
 \sum_{i=1}^n \bigl(2+ 2a_i^2(s+4) \bigr) + (2s+8)\Bigl(\sum_{i<j}k_i k_ja_i a_j + \sum_{i=1}^n {k_i \choose 2} a_i^2\Bigr)+\sum_{i=1}^n {k_i \choose 2} - (m-1) ,
 $$
where the first sum is the moduli count for the universal deformation spaces $\mbox{Def}(E_i)$, whereas the other terms account for the dimension of the isomorphism classes of the $m-1$ successive extensions. It is not difficult to check that this quantity is smaller that
$$
 2 + 2(s+4) \left( \sum_{i=1}^n k_i^2a_i^2 +  2\sum_{i<j}k_ik_ja_ia_j \right)= 2+2a^2(s+4)
$$
which is the dimension of the moduli space of simple Ulrich bundles of rank $2a$.  Thus, there exist simple Ulrich bundles
of rank
 $2a$ which are stable.
\end{proof}

\vskip 4pt

\subsection{The Chow form of a $K3$ surface.} Let $S\subset \PP^{s+1}$ be a $K3$ surface for which the hypothesis of Theorem \ref{thm: ulrich} applies. We choose a special rank two Ulrich bundle $E$ on $S$, hence $\mbox{det}(E)=\OO_S(3)$ and $h^0(S, E)=4s$. We have defined in the introduction the exterior algebras ${\bf \Lambda}:=\bigwedge^{\bullet} H^0(S, \OO_S(1))^{\vee}$ and ${\bf \Lambda}^{\vee}:=\bigwedge^{\bullet} H^0(S, \OO_S(1))$, with gradings $${\bf \Lambda}_{-i}:=\bigwedge ^i H^0(S, \OO_S(1))^{\vee} \mbox{ and } \ {\bf \Lambda}^{\vee}_i:=\bigwedge ^i H^0(S, \OO_S(1))$$
respectively. To $E$, viewed as a sheaf on $\PP^{s+1}$, one associates a minimal bi-infinite exact sequence of free graded ${\bf \Lambda}$-modules called the \emph{Tate resolution}:
$$T^{\bullet}(E): \cdots \rightarrow T^{-2}(E)\rightarrow T^{-1}(E)\rightarrow T^0(E)\rightarrow T^1(E)\rightarrow T^2(E)\rightarrow \cdots$$
It is shown in \cite{EFS} that one has isomorphisms of graded ${\bf \Lambda}$-modules
$$T^p(E)=\bigoplus_{i=0}^2 {\bf \Lambda}^{\vee}(i-p)\otimes H^i(S, E(p-i)).$$
For an Ulrich bundle, the Tate resolution is particularly simple, for instance
$$T^{-1}(E)={\bf \Lambda}^{\vee}(3)\otimes H^2(S, E(-3)) \ \mbox{ and } \ \ T^0(E)={\bf \Lambda}^{\vee}\otimes H^0(S, E).$$
To pass from the Tate resolution of $E$ to the Chow form $\cZ(S)$ one applies the functor ${\bf U}_3$ of \cite{ES} from the category of free graded ${\bf \Lambda}$-modules to that of vector bundles over the Grassmannian $\GG:=G\bigl(s-1, H^0(S, \OO_S(1))^{\vee}\bigr)$ of projective codimension $3$ planes in $\PP^{s+1}$. This functor replaces the module ${\bf \Lambda}^{\vee}(p)$ by the $p$-th exterior power of the rank $3$ tautological bundle $\cU$ on $\GG$. The resulting complex ${\bf U}_3 T^{\bullet }(E)$ consists of a single morphism of vector bundles over $\GG$
$${\bf U}_3(\varphi):  H^2(S, E(-3))\otimes \bigwedge^3 \cU \rightarrow H^0(S, E)\otimes \OO_{\GG}.$$
The determinant of $\varphi$ gives the equations in Pl\"ucker coordinates of the Chow form of $S$, see \cite{ES} Corollary 3.4. In what follows we describe explicitly the linear map
\begin{equation}\label{cech}
\varphi: \bigwedge^3 H^0(S, \OO_S(1))\otimes H^2(S, E(-3))\rightarrow H^0(S, E).
\end{equation}

\vskip 3pt
\noindent \emph{Proof of Theorem \ref{chow}.}
We fix sections $x^1, x^2, x^3\in H^0(S, \OO_S(1))$ and $u\in H^2(S, E(-3))$. Let $\cV=\{V_{\alpha}\}_{\alpha}$ be a covering of $S$ and $\{u_{\alpha \beta \gamma}\}\in \check{Z}^2(\cV, E(-3))$
a \v{C}ech cocycle representative of $u$. Since $H^2(S, E(-2))=0$, we obtain that $\bigl\{x^i\cdot u_{\alpha \beta \gamma}\bigr\} \in \check{Z}^2(\cV, E(-2))=\check{B}^2(\cV, E(-2))$.
In particular, there exist $1$-cocycles $\{t_{\alpha \beta}^i\}\in \check{C}^1(\cV, E(-2))$ such that the following hold for $i=1,2,3$:
$$x^i\cdot u_{\alpha \beta \gamma}=t_{\alpha \beta}^i-t_{\beta \gamma}^i+t_{\gamma \alpha}^i\in \Gamma(U_{\alpha \beta \gamma}, E(-2)).$$
Since $H^1(S, E(-1))=0$, we find $\{x^i\cdot t_{\alpha \beta}^j-x^j\cdot t_{\alpha \beta}^i\}\in \check{Z}^1(\cV, E(-1))=\check{B}^1(\cV, E(-1))$, therefore for $i\neq j$ there exist cocycles $\{v^{ij}_\alpha\}\in \check{C}^0(\cV, E(-1))$, such that
$$ x^i\cdot t_{\alpha \beta}^j-x^j\cdot t_{\alpha \beta}^i=v^{ij}_{\alpha}-v^{ij}_{\beta}.$$ Therefore we obtain a section $w\in H^0(S, E)$ such that
$$w=x^1v^{23}_{\alpha}-x^2v_{\alpha}^{13}+x^3v_{\alpha}^{12}=x^1v^{23}_{\beta}-x^2v_{\beta}^{13}+x^3v_{\beta}^{12}\in \Gamma(V_{\alpha \beta}, E).$$
We set $\varphi(x_1\otimes x_2 \otimes x_3\otimes s):=w$. Clearly this construction vanishes on symmetric tensors, hence it gives rise to the map (\ref{cech}), which is the degree zero part of the ${\bf \Lambda}$-module map $$\Bigl\{\varphi^i:\bigwedge^{i+3} H^0(S, \OO_S(1))\otimes H^2(S, E(-3))\longrightarrow \bigwedge ^i H^0(S, \OO_S(1))\otimes H^0(S, E)\Bigr\}_{i}$$ appearing in the Tate resolution of $E$. It is straightforward to check that the image of $\{\varphi^i\}_i$ is precisely the kernel of the Koszul differential
$T^0(E)\rightarrow T^1(E)$, which finishes the proof. \hfill $\Box$

\begin{rmk}
One may also view the above constructed ${\bf \Lambda}$-linear differential as a linear map $\varphi:H^0(S, E)^{\vee}\otimes H^0(S, E)^{\vee}\rightarrow \bigwedge^3 H^0(S, \OO_S(1))^{\vee}$, given by a $4s\times 4s$ matrix of linear forms in the Pl\"ucker coordinates of $\GG$. This matrix is antisymmetric and its pfaffian gives the equation for $\cZ(S)$.
\end{rmk}

\section{Ulrich bundles and the Minimal Resolution Conjecture}

The aim of this closing section is to explain how the existence of a rank $2$ Ulrich bundle on a $K3$ surface $S$ determines the Koszul cohomology groups of some $0$-cycles on $S$ and  in particular implies that these cycles verify a MRC type condition.  Following \cite{G}, for a sheaf $\F$ and a line bundle $L$ on a projective variety $X$, we form the graded $\SS:=\mbox{Sym } H^0(X, L)$-module $\Gamma_X(\F, L):=\bigoplus_{j\in \mathbb Z} H^0(X, \F\otimes L^{\otimes j})$ and denote
$$K_{i, j}(X;\F, L):=\mbox{Tor}^S_i(\Gamma_X(\F, L), \mathbb C)_{i+j}.$$
 Note that with this notation, one has that $K_{i, j}(X, L)=K_{i, j}(X; \OO_X, L)$.

 \vskip 3pt

We consider a rank $2$ Ulrich bundle $E$ on a $K3$ surface $S$ constructed in Theorem \ref{thm: ulrich} and for $n\geq 1$, we set   $E_n:=E(n)$. The new vector bundle is globally generated and has Chern classes $c_1(E_n)=\mathcal O_S(2n+3)$ and $c_2(E_n)=\gamma_n:=2n^2s+6sn+5s+4$ respectively. We choose a general section $\sigma\in H^0(S,E_n)$ and consider the associated $0$-dimensional subscheme $\Gamma=\Gamma_n\subset S$, together with the corresponding short exact sequence:
\[
0\longrightarrow \mathcal O_S(-2n-3)\longrightarrow E_n^{\vee}\longrightarrow \mathcal I_{\Gamma/S}\longrightarrow 0.
\]

Note that $E_n^{\vee}\cong E(-n-3)$. Since $h^1(S,\mathcal O_S(m))=0$ for all $m$, we have an induced short exact sequence of graded modules over $\SS:=\mbox{Sym } H^0(\mathcal O_S(1))$
\[
0\longrightarrow  \bigoplus_j H^0(S, \mathcal O_S(j-2n-3))\longrightarrow
\bigoplus_jH^0(S, E(j-n-3))\longrightarrow  \bigoplus_j H^0(I_{\Gamma/S}(j))\longrightarrow 0,
\]
which, after applying \cite{G} Corollary 1.d.4, yields to a long exact sequence for Koszul cohomology groups:
\[
\cdots\rightarrow K_{i,j-n-3}(S; E,H)\rightarrow K_{i,j}(\PP^{s+1}; \mathcal I_{\Gamma/S},\mathcal O_{\PP^{s+1}}(1))\rightarrow K_{i-1,j-2n-2}(S,H)\to
\]
\[
\to K_{i-1,j-n-2}(S; E,H)\rightarrow K_{i-1,j+1}(\PP^{s+1}; \mathcal I_{\Gamma/S},\mathcal O_{\PP^{s+1}}(1))
\rightarrow K_{i-2,j-2n-1}(S,H)\rightarrow \cdots
\]

As $E$ is an Ulrich bundle, $K_{p,q}(S; E,H)=0$ for all $q\ne 0$, see \cite{ES} Proposition 2.1. 
Furthermore, using \cite{FMP} Theorem 1.2,  for $j\geq \mbox{reg}(S)=4$ we obtain isomorphisms
\begin{equation}\label{fmp}
K_{i, j}(\Gamma, \OO_{\Gamma}(1))\cong K_{i-1, j+1}\bigl(\PP^{s+1}; \I_{\Gamma/S}, \OO_{\PP^{s+1}}(1)\bigr).
\end{equation}
We show that the condition that $E$ be an Ulrich bundle on $S$, translates into a MRC type vanishing condition for the Koszul cohomology groups of the scheme $\Gamma$.

\begin{prop}\label{ulmrc}
Let $i, a, n\geq 1$ and $\Gamma\subset S$ the zero locus of a general section of the vector bundle $E_n$. The following isomorphisms hold:
$$K_{i, n+3+a}(\Gamma, \OO_{\Gamma}(1))=K_{i-2, a+2-n}(S, H) \ \mbox{ and }  \ K_{i+1, n+2+a}(\Gamma, \OO_{\Gamma}(1))=K_{i-1,a+1-n}(S, H).$$
If, furthermore, $S$ is a $K3$ surface with $\mathrm{Pic}(S)=\mathbb Z\cdot [H]$, then
$$b_{i, n+3+a}(\Gamma)\cdot b_{i+1, n+2+a}(\Gamma)=0.$$
\end{prop}
\begin{proof} The isomorphisms follow from the previous exact sequence and (\ref{fmp}) by substituting $j:=n+3+a$. The second part is an application of the first, coupled with the description of the Koszul cohomology of a $K3$ surface given by Voisin \cite{V} in the course of her  proof of Green's conjecture, see also \cite{AF11} Theorem 1.3 for a precise formulation of what is being used here. Indeed, for any smooth polarized $K3$ surface $(S,H)$ of genus $g$, if $c$ denotes the Clifford index of all smooth curves in $|\OO_S(H)|$, then the following equivalences are established in \emph{loc. cit.}:
$$K_{p,1}(S,H)\neq 0\Leftrightarrow 1\leq p\leq g-c-2\ \mbox{ and }\ \ K_{p-1,2}(S,H)=0\Leftrightarrow c+1\leq p\leq g.$$
If $\mbox{Pic}(S)=\mathbb Z\cdot [H]$, then of course $c=\lfloor \frac{g-1}{2}\rfloor$, in particular $c+1>g-c-2$. To use the terminology of \cite{CEFS}, it follows that the resolution of
the $K3$ surface is \emph{natural}, that is, the product of any two consecutive Betti numbers on a diagonal of the Betti table is equal to zero.
\end{proof}

%

\begin{rmk} Suppose $S$ is a $K3$ surface and let  $\Gamma\subset S\subset \PP^{s+1}$ be the zero locus of a general section of the vector bundle $E_n$ as above. Proposition \ref{ulmrc} gives a complete picture of the minimal resolution of $\Gamma$, in the spirit of the Minimal Resolution Conjecture. Indeed, for $j<n+3$, one has $h^0(\mathcal I_{\Gamma/S}(j))=0$ which implies $h^0(\mathcal I_{\Gamma}(j))=h^0(\mathcal I_{S}(j))$ and furthermore $K_{i,j}(S,H)=K_{i,j}(\Gamma,\mathcal O_\Gamma(1))$ for all $i$ and for $j\ge n+3$. In particular, via Proposition \ref{ulmrc}, if $\mathrm{Pic}(S)=\mathbb Z\cdot [H]$, then the Betti numbers of $\Gamma$ verifiy the condition  $b_{i+1,j-1}(\Gamma)\cdot b_{i,j}(\Gamma)=0$ for all $i$ and $j$.

\vskip 3pt

However, for large $n$ there is a distinction  between the minimal resolution of $\Gamma$ and that of $\gamma_n$ general points of $S$. The Hilbert polynomial of $S$ is given by $P_S(n)=n^2s+2$. Set $u_n:=\left[\sqrt{\frac{\gamma_n-2}{s}}\right]+1$, so that $P_S(u_n-1)\le \gamma_n< P_S(u_n)$. MRC for a set $Z$ of  $\gamma_n$ general points on $S$ predicts that $b_{i+1,u_n-1}(Z)\cdot b_{i,u_n}(Z)=0$; at the same time, there exists $i$ such that $b_{i,u_n-1}(Z)\ne 0$, see \cite{FMP}, Theorem 1.2. On the other hand, for large $n$, Proposition \ref{ulmrc} implies that $b_{i,u_n-1}(\Gamma)=b_{i,u_n}(\Gamma)=0$ for all  $i$.

\vskip 3pt

The explanation for this discrepancy lies in the fact that the Hilbert function of $\Gamma$ is different from that of $\gamma_n$ general points on $S$. Indeed, as observed in \cite{FMP}, Theorem 1.2, for a general  $Z\in S^{[\gamma_n]}$ the equality $h^0(\mathcal I_Z(u_n-1))=h^0(\mathcal I_S(u_n-1))$ holds, whereas $$h^0(\mathcal I_\Gamma(u_n-1))-h^0(\mathcal I_S(u_n-1))=h^0(\mathcal I_{\Gamma/S}(u_n-1))=h^0(S, E(u_n-n-4))\ne 0,$$ as $u_n\geq n+4$. Since graded Betti numbers vary semicontinuously only in families of subschemes having the same \emph{Hilbert function}, one cannot pass from the resolution of $\Gamma$ to that of a general $0$-cycle of length $\gamma_n$ on $S$.

\end{rmk}

\end{document}